\documentclass[11pt]{article}

\usepackage{tgpagella}
\usepackage[utf8]{inputenc}
\usepackage[T1]{fontenc}
\usepackage{amssymb}
\usepackage{amsfonts}
\usepackage{amsmath}
\usepackage{amsthm}
\usepackage{mathrsfs}
\usepackage{comment}

\newcommand{\PA}{\textnormal{PA}}

\newcommand{\set}[2]{\lbrace #1 \ \mid \ #2 \rbrace}
\newcommand{\res}{\upharpoonright}
\newcommand{\CT}{\textnormal{CT}}
\newcommand{\df}[1]{\textbf{#1}}
\newcommand{\num}[1]{\underline{#1}}
\newcommand{\La}{\mathscr{L}}
\newcommand{\ind}{\textnormal{ind}}
\newcommand{\LPA}{\La_{\PA}}
\newcommand{\ElDiag}{\textnormal{ElDiag}}
\newcommand{\IDelta}{\textnormal{I} \Delta}

\newcommand{\Form}{\textnormal{Form}}
\newcommand{\Term}{\textnormal{Term}}
\newcommand{\ClTerm}{\textnormal{ClTerm}}
\newcommand{\TermSeq}{\textnormal{TermSeq}}
\newcommand{\ClTermSeq}{\textnormal{ClTermSeq}}
\newcommand{\Sent}{\textnormal{Sent}}
\newcommand{\val}[1]{{#1}^{\circ}}
\newcommand{\Var}{\textnormal{Var}}

\newcommand{\Comp}{\textnormal{Comp}}
\newcommand{\Val}{\textnormal{Val}}
\newcommand{\FV}{\textnormal{FV}}
\newcommand{\sd}{\textnormal{dpt}}
\newcommand{\SRP}{\textnormal{SRP}}

\newtheorem{theorem}{Theorem}
\newtheorem{proposition}{Proposition}
\newtheorem{lemma}{Lemma}

\theoremstyle{definition}
\newtheorem{definition}{Definition}
\newtheorem{remark}{Remark}
\newtheorem{example}{Example}

\theoremstyle{plain}
\newtheorem*{lemma*}{Lemma}

\title{Disjunctions with stopping conditions}
\author{Roman Kossak, Bartosz Wcisło}

\begin{document}

	\maketitle

\begin{abstract}
	We introduce a tool for analysing models of $\CT^-$, the compositional truth theory over Peano Arithmetic. We present a new proof of Lachlan's theorem that the arithmetical part of models of $\CT^-$ are recursively saturated. We also use this tool to provide a new proof of theorem from \cite{lelykwcislomodels} that all models of $\CT^-$ carry a partial inductive truth predicate. Finally, we construct a partial truth predicate defined for a set of formulae whose syntactic depth forms a nonstandard cut which cannot be extended to a full truth predicate satisfying $\CT^-$.
\end{abstract}

\section{Introduction}
In 1979, Alistair Lachlan visited Warsaw. There, together with Henryk Kotlarski and Stanisław Krajewski, he worked on nonstandard satisfaction classes in models of arithmetic, and, in particular, he proved that a model that admits a full satisfaction class must be recursively saturated \cite{lachlan}. The result  is an easy observation if one assumes in addition that the satisfaction class is inductive, but it was quite surprising that the result holds without that assumption, and the proof was highly original. Since then, the proof has been simplified somewhat, but still its standard presentation, such as the one in \cite{kaye},  involves some seemingly necessary technicalities. In this paper, we give a proof of Lachlan's theorem that is essentially the standard one, but before giving the proof,  we isolate the part of the argument, that can be dubbed Lachlan's trick, and present it as a specific tool that is later used to prove other results. That tool---disjunctions with stopping condition---is presented in Section 3, after an example that motivates the definition, followed by Section 4, in which we give a proof of Lachlan's theorem.

The original proof of Lachlan had a reputation for lacking any initial motivation and for being very difficult to grasp on the intuitive level. One of our prime aims in this paper is to present Lachlan's argument not as an isolated and \textit{ad hoc} trick, but as a clearly motivated and reusable technique. 

Lachlan's proof and some of its consequences were analyzed by Stuart Smith in this Ph. D.~thesis \cite{smith_phd}. In particular, Smith showed that if $S$ is a full satisfaction class on a model $M$ of $\PA$, then there is an undefinable class $X$ of $M$ that is definable in $(M,S)$ \cite{smith}. That result  shows that rather classless, recursively saturated models of $\PA$ do not admit full satisfaction classes. In Section 5, we use disjunctions with stopping condition to prove a strengthening of Smith's theorem. We show that one can always find an $X$ as above that is an inductive partial satisfaction class. This result has been already published in \cite{lelykwcislomodels}, but the techniques discussed in this paper allowed us to obtain a significantly simpler and cleaner proof which allows us to avoid many technicalities and makes clear the analogy to the original proof of Lachlan's theorem.

The results of sections 3, 4 and 5 are due to the second author. They are a part of his Ph.D.~thesis \cite{bartek}.

In Section 6, we consider a model theoretic question concerning extendability of nonstandard satisfaction classes. Kotlarski, Krajewski, and Lachlan proved that every  countable recursively saturated model admits a full satisfaction class. A new, model theoretic proof of this result was given by Ali Enayat and Albert Visser in \cite{enayatvisser}. This new proof renewed interest in a more detailed study of the variety  of nonstandard satisfaction classes on countable, recursively saturated models of $\PA$. In particular, if $S$ is a satisfaction class on a model $M$, and $N$ is a recursively saturated elementary end extension of $M$, one is interested in conditions that imply that $S$ can be extended to a satisfaction class of $N$. In Section 6, we construct a slightly pathological example showing an obstruction to proving a desired general theorem about existence of such extensions. This part of the paper is our joint work.

\section{Preliminaries}

In this section, we list basic technical definitions and facts which we use in our paper. 

This work concerns extensions of Peano Arithmetic ($\PA$). All basic facts concerning $\PA$ (including coding) and its models may be found, e.g. in \cite{kaye}. We assume that Peano Arithmetic is formulated in a language $\LPA$ with one unary function symbol $S(x)$ (the successor function) and two binary function symbols $+$ and $\times$. We assume that the reader is acquainted with arithmetisation of syntax. We will use the following notation:
\begin{itemize}
	\item $\Var(x)$ is a formula which defines the set of (G\"odel codes of) first order variables.
	\item $\Term_{\PA}(x)$ is a formula which defines the set of arithmetical terms. $\ClTerm(x)$ defines the set of closed arithmetical terms. $\TermSeq_{\PA}(x)$ defines sequences of arithmetical terms. $\ClTermSeq_{\PA}(x)$ defines sequences of closed arithmetical terms.
	\item $\Form_{\PA}(x)$ is a formula which defines the set of arithmetical formulae. $\Form^{\leq 1}_{\PA} (x)$ represents the set of arithmetical formulae with at most one free variable.
	\item $\Sent_{\PA}(x)$ is a formula which defines the set of arithmetical sentences.
	\item We will use expressions such as $x \in \Form_{\PA}, x \in \Sent_{\PA}$, or $x \in \Term_{\PA}$ interchangeably with $\Form_{\PA}(x), \Sent_{\PA}(x)$, or $\Term_{\PA}(x)$. In other words, our notation will conflate formulae defining certain sets and those sets themselves.  
	\item For $\phi \in \Form_{\PA}$, $\FV(\phi)$ is a formula defining the set of free variables of $\phi$ and $\Val (\phi)$ is a formula defining the set of valuations, i.e. finite functions, whose domains contain all free variables of $\phi$.   
	\item $y = \num{x}$ is a binary formula which defines the relation "$y$ is the numeral denoting $x$," i.e., the numeral $S\ldots S0$, where $S$ occurs $x$ times. We will actually use $\num{x}$ as if it were a term and write expressions such as $\forall x T\phi(\num{x})$ to denote: "for all $x$, $T$ holds of the effect of substituting $\num{x}$ for the only free variable in the formula $\phi$."
	\item $x_y = z$ is a ternary formula which defines the relation "$y$-th element of the sequence $x$ is $z$." We will actually use this relation in a functional way. For instance, we will use an expression $a_c$ for a sequence $a$, as if $a_c$ were a term.
	\item $\val{x}=y$ is a binary formula representing the relation: "$y$ is the value of the term $x$." E.g., $\PA \vdash \val{(\num{x} + S0)} = S(x)$. We will use $\val{x}$ as if it were a term. If $\bar{s} \in \ClTermSeq_{\PA}$, then by $\bar{\val{s}}$ we mean the sequence of values of terms in $s$. 
	\item If $\phi \in \Form_{\PA}$, then $\sd(\phi)$ denotes the syntactic depth of $\phi$, that is, the maximal number of quantifiers and connectives on a path in the syntactic tree of $\phi$.
\end{itemize}  

In the paper, we discuss models of a theory $\CT^-$ and related theories. $\CT^-$ is an axiomatisation of compositional truth predicate for arithmetical sentences. Its language is $\LPA$ together with a unary predicate $T(x)$ with the intended reading "$x$ is a (G\"odel code of a) true sentence." Its axioms are axioms of $\PA$ together with the following ones: 
\begin{enumerate}
	\item $\forall s,t \in \ClTerm_{\PA} \ \ T(s=t) \equiv (\val{s} = \val{t}).$
	\item $\forall \phi \in \Sent_{\PA}  \ \ T\neg \phi \equiv \neg T \phi.$
	\item $\forall \phi, \psi \in \Sent_{\PA} \ \ T \phi \vee \psi \equiv T \phi \vee T \psi.$
	\item $\forall v \in \Var \forall \phi \in \Form^{\leq 1}_{\PA} \ \ T \exists v \phi \equiv \exists x T \phi[\num{x}/v].$
	\item $\forall \bar{s},\bar{t} \in \ClTermSeq_{\PA} \forall \phi \in \Form_{\PA} \ \ \val{\bar{s}} = \val{\bar{t}} \rightarrow T\phi(\bar{t}) \equiv T \phi(\bar{s})$.
\end{enumerate}

The last item is called the \df{regularity axiom}. Although it is not essential to the present paper (all theorems still hold if we drop the axiom), we include it nevertheless, since truth theories without induction can display certain pathologies which add a layer of technical complexity to the considerations. For instance, in the absence of the regularity axiom, we cannot deduce that $T\exists v \phi(v)$ holds from the fact that $T \phi(0 +0)$ holds for a nonstandard $\phi$, since in the axiom for the existential quantifier we explicitly require that $\phi$ is witnessed by a numeral.

We will also consider some variants of $\CT^-$. 

\begin{definition} \label{def_restricted_ctminus}
Let $I(x)$ be a unary predicate which will play a role of a definition of a cut. By $\CT^- \res I$ we mean $\CT^-$ in which the compositional axioms are only assumed to hold for formulae whose depth is in this cut, but with no restriction on the size of terms, i.e.:

\begin{enumerate}
	\item $I(x)$ defines a cut.
	\item $\forall s,t \in \ClTerm_{\PA} \ \ T(s=t) \equiv (\val{s} = \val{t}).$
	\item $\forall \phi \in \Sent_{\PA}  \ \ \Big(\sd(\neg \phi) \in I \rightarrow T\neg \phi \equiv \neg T \phi \Big).$
	\item $\forall \phi, \psi \in \Sent_{\PA}  \ \ \Big(\sd(\phi \vee \psi) \in I \rightarrow T (\phi \vee \psi) \equiv T \phi \vee T \psi \Big).$
	\item $\forall v \in \Var \forall \phi \in \Form^{\leq 1}_{\PA} \ \ \Big(\sd(\exists v \phi) \in I \rightarrow T \exists v \phi \equiv \exists x T \phi[\num{x}/v] \Big).$
	\item $\forall \bar{s},\bar{t} \in \ClTermSeq_{\PA} \forall \phi \in \Form_{\PA} \ \ \val{\bar{s}} = \val{\bar{t}} \rightarrow T\phi(\bar{t}) \equiv T \phi(\bar{s})$.
\end{enumerate}

If $c \in M \models \PA$, we define $\CT^- \res c$ in an analogous way with a constant $c$ instead of $I(x)$ and with formulae and sentences restricted to $[0,c]$ instead of the cut $I$ (alternatively, we could view $c$ as a fresh constant). 
\end{definition}

Notice that in $\CT^-$ there are no induction axioms for the formulae containing the truth predicate (induction for arithmetical formulae is assumed, as $\CT^-$ and its variations are extensions of $\PA$). If we extend our theories with full induction, we denote them with $\CT$ or $\CT \res c$.

When dealing with truth predicates restricted to certain syntactic depth, it proves handy to introduce an additional technical regularity condition. 

\begin{definition} \label{def_structural_equivalence}
	Let $\phi, \psi \in \Sent_{\PA}$. We say that $\phi, \psi$ are \df{structurally equivalent} if there exists a formula $\eta \in \Form_{\PA}$ and sequences of closed terms $\bar{s}, \bar{t} \in \ClTermSeq_{\PA}$ such that:
	\begin{itemize}
		\item 	 $\bar{\val{s}} = \bar{\val{t}}$.
		\item   $\eta(\bar{s})$ differs from $\phi$ by renaming bound variables in such a way that distinct variables remain distinct.
		\item   $\eta(\bar{t})$ differs from $\psi$ by renaming bound variables in such a way that distinct variables remain distinct.
	\end{itemize}
If $\phi$ and $\psi$ are structurally equivalent, we denote it by $\phi \simeq \psi$. 
\end{definition}

\begin{example}
	The following two sentences are structurally equivalent:
	\begin{eqnarray*}
		\phi_1 & = & \exists x \big( x +0 =  S(S0) \big) \\
		\phi_2 & = & \exists y \big( y + (0 \times S0) = S0 + (S0 \times S0) \big).
	\end{eqnarray*}
\end{example}

By convention, we will also use the expression $\phi \simeq \psi$ to denote the formalised arithmetised statement that $\phi$ and $\psi$ are structurally equivalent. Finally, we define the desired regularity property.
\begin{definition} \label{def_structural regularity property}
	By \df{strucutral regularity property} ($\SRP$), we mean the following axiom:
	\begin{displaymath}
	\forall \phi, \psi \in \Sent_{\PA} \Big( \phi \simeq \psi \rightarrow T\phi \equiv T \psi \Big).
	\end{displaymath}
\end{definition}

\section{Introducing disjunctions with stopping conditions}

In this section, we describe the main tool of our paper. The technique of disjunctions with stopping conditions involves a
propositional construction which essentially allows us to express infinite definitions by cases under the truth predicate. They have been first explicitly defined in \cite{cieslinski_lelyk_wcislo}, but in fact they were used much earlier by Smith in \cite{smith}. The idea on which they are based can be traced back to \cite{lachlan}. Since the construction of disjunctions with stopping condition is rather intricate, let us begin with an intuitive description of how they work.

Let $(M,T)$ be a model of $\CT^-$ and let $p = (\phi_i)_{i\in\omega}$ be a computable type in one variable and with finitely many parameters in the arithmetical language that is finitely realisable in $M$. We will try to show that this type is realised in $M$. One obvious strategy would be as follows. Let $(\phi_i)_{i<c}$ be a nonstandardly long coded sequence in $M$ which prolongs $p$. For $a<c$, let 
\begin{displaymath}
\beta_a(x) : = \bigwedge_{i\leq a} \phi_i(x).
\end{displaymath}
Notice that since $p$ is a type, for any standard $j$ we have:
\begin{displaymath}
(M,T) \models \exists x \ \ T\beta_j(\num{x}).
\end{displaymath}
The goal is to show that for some nonstandard $b \in M$, 
\begin{displaymath}
(M,T) \models \exists x \ \ T\beta_b(\num{x}).
\end{displaymath}
Then, using compositional axioms, we could show that any such $b$ realises $p$. Unfortunately, it is not really clear, how to ensure the existence of such $b$ in total absence of  induction for the truth predicate (otherwise, we could use an easy overspill argument).

In essence, we would like to define the set of elements satisfying a given type  using a nonstandard formula. Now, an extremely clever observation by Lachlan which is one of the central ingredients of his proof is that we do not have to use induction to obtain a formula which defines the set of elements realising a certain type. Let us describe this in more detail.

For a fixed type $p$, we introduce a notion of rank. The rank $r$ of a formula $\psi \in \Form(M)$  measures how close the elements $x$ satisfying $(M,T) \models T \psi(\num{x})$ come to  satisfying the type $p$. This can be defined as follows: if $\psi$ is not satisfied by any element or  $(M,T) \models \exists x \ \psi(x) \wedge \neg \phi_0(x)$, this is very bad and we set $r(\psi) = - \infty$. If there are elements such that $(M,T) \models T\psi(\num{x})$ and any such $x$ happens also to satisfy $\phi_0(x), \ldots, \phi_{n}(x)$, but not $\phi_{n+1}(x)$, we set rank $r(\psi) = n$. If any element defined by $\psi$ realises the whole type, then we set $r(\psi) = \infty$.  Notice that  the formulae $\beta_n(x)$ defined above  have rank at least $n$.

Now our task may be reformulated as follows: find a formula whose rank is  $\infty$. It turns out that this may be obtained without using induction thanks to the following  lemma that is implicit in the work of Lachlan.

\begin{lemma}\label{lem_dobre_porzadki_w_M}
	Let $W$ be a well order with a maximal element, let $M \models \PA$ be a nonstandard model and let $f: M \to W$ be a function such that for any $x \in M$:
	\begin{itemize}
		\item either $f(x)$ is the maximal element of $W$;
		\item or $f(x+1) > f(x)$.
	\end{itemize}
	Then there exists $x \in M$ such that $f(x)$ is the maximal element of $W$. 
\end{lemma}

\begin{proof}
	Let $W,M,f$ satisfy the assumptions of the theorem. Suppose that there is no $x \in M$ such that $f(x)$ is maximal in $W$. Pick any nonstandard $a \in M$. Then 	
	\begin{displaymath}
	f(a) > f(a-1) > f(a-2) > \ldots 
	\end{displaymath} 
	is an infinite descending $\omega$-chain in  $W$. Contradiction.
\end{proof}

Now we will describe a na\"ive attempt to use Lemma \ref{lem_dobre_porzadki_w_M}, applied to the order $\{-\infty\} \cup \omega \cup \{\infty\} $, to find an element realising $p$.  To this end, for a given formula $\psi$, we will define in a uniform way another formula of a higher rank. 

It is easy to see that for any formula $\psi(x)$, there is a sentence $\alpha_n[\psi]$ which expresses that $\psi$ has rank $n-1$ for $n>0$ or rank $-\infty$ for $n=0$  (the details are in the proof of Lemma \ref{lem_lemat_o_randze_dla p_rangi} in the next section).

Let $\gamma_0$ be $x=x$. Then, given  $\gamma_a$ we define $\gamma_{a+1}$ as follows:
\begin{displaymath}
\gamma_{a+1} := \big(\alpha_0[\gamma_a] \wedge \beta_0(x) \big) \vee \big(\alpha_1[\gamma_a] \wedge \beta_1(x)\big) \vee  \ldots \vee \big(\alpha_c[\gamma_a] \wedge \beta_c(x)\big)
\end{displaymath}
with parentheses grouped to the left. 

Read $\gamma_a$ as a definition by cases: either $\gamma_a$ has rank smaller than $0$, i.e. $-\infty $, and $\beta_0(x)$ or $\gamma_a$ has rank $0$ and $\beta_1(x)$, or $\gamma_{a}$ has rank $1$ and $\beta_2(x)$ etc. Na\"ively, for any $a$ the formula $\gamma_{a+1}$ should have higher rank than $\gamma_{a}$. Namely, if $\gamma_{a}$ has rank $n$, then the only formula $\alpha_j[\gamma_{a}]$ which can be true is $\alpha_{n+1}[\gamma_{a}]$. Then the whole formula $\gamma_{a+1}$ is equivalent over propositional logic to $\beta_{n+1}(x)$ which has rank at least $n+1$ (the case where the rank of $\gamma_{a}$ is equal to $-\infty$ is handled in a similar fashion). This, coupled with Lemma \ref{lem_dobre_porzadki_w_M} would ensure the existence of a formula with rank $\infty$.

  Unfortunately, this definition does not work correctly. This is because infinite conjunctions and disjunctions may behave badly in general models of $\CT^-$. Even if $\gamma_a$ indeed has rank $n$,  $\gamma_{a+1}$ may still define the whole model $M$. Consequently, its rank can be even $- \infty$ if $\phi_0$ defines any nontrivial subset of the model. Even if we fix an $x$ such that $\neg \phi_0(x)$ holds, we still we might have:
\begin{displaymath}
(M,T) \models T \big(\alpha_0[\gamma_a] \wedge \beta_0(\num{x}) \big) \vee \big(\alpha_1[\gamma_a] \wedge \beta_1(\num{x})\big) \vee  \ldots \vee \big(\alpha_c[\gamma_a] \vee \beta_c(\num{x})\big)
\end{displaymath}  
Our truth predicate will be able to recognise:
\begin{displaymath}
(M,T) \models \neg T \alpha_{0}[\gamma_a]
\end{displaymath}
and consequently it will yield the first disjunct false. In a similar fashion, it can yield the second disjunct false, the third disjunct false etc. However, it will not be able to conclude that the whole disjunction is false. 

We can in fact guarantee that in a typical model of $\CT^-$, nonstandardly large disjunctions will produce this kind of pathological behaviour.  In \cite{enayat_pachomow}, it is shown that $\CT^-$ enriched with the principle: "a finite disjunction is true iff one of the disjuncts is true" is not conservative over $\PA$ and in fact has the same arithmetical strength as $\CT_0$, a compositional truth theory $\CT^-$ with $\Delta_0$ induction for the formulae in the extended language.\footnote{As shown in \cite{lelykphd}, preceded by a closely related result in \cite{kotlarski}, the arithmetical strength of this theory can be characterised as $\omega$ iterations of the uniform arithmetical reflection over $\PA$.}

Now, a disjunction with stopping condition is a propositional construction which allows us to do exactly what we have failed to do in our  na\"ive attempt above. In other words, we can define a nonstandard arithmetical formula $\gamma_{a+1}(x)$ such that if $k \in \omega$ is the least number for which $(M,T) \models \alpha_{k+1}[\gamma_a]$ (that is, $\gamma_a$ has rank $k+1$), then
 \begin{displaymath}
(M,T) \models \forall x \ \ \big(T\gamma_{a+1}(\num{x}) \equiv \beta_{k+1}(x) \big).
\end{displaymath}

The definition of  such $\gamma_{a+1}(x)$ which will be given in the proof of Lachlan's theorem in the next section uses a particular instance of a disjunction with  with a stopping condition as defined  below.  Roughly, to check if  $\gamma_{a+1}(x)$ holds, we ask if the rank of $\gamma_{a}(x)$ is below 0, if yes, we check if $\beta_0(x)$ holds. If yes, our job is done, if not we ask if the rank of $\gamma_{a}(x)$ is below 1, and  if yes, we check if $\beta_1(x)$ holds. If yes, we stop, otherwise we continue. If we get to $\beta_c(x)$ without stopping, we declare that $\gamma_{a+1}(x)$ does not hold. \begin{definition}
	Let $c \in M$, and let $(\alpha_i)_{i\leq c}$, $(\beta_i)_{i\leq c}$ be coded sequences of sentences of $M$. Then we define a \df{disjunction with stopping condition} $\alpha$ 
	\begin{displaymath}
		\bigvee_{i=a}^{c, \alpha} \beta_i
	\end{displaymath} 
	by backwards induction on $k$.
	\begin{itemize}
		\item $\bigvee_{i=c}^{c, \alpha} \beta_i = (\alpha_c \wedge \beta_c)$.
		\item $\bigvee_{i=a}^{c,\alpha} \beta_i = (\alpha_a \rightarrow \beta_a) \wedge [(\alpha_a \wedge \beta_a)  \vee (\neg \alpha_a \wedge \bigvee_{i=a+1}^{c, \alpha} \beta_i )]$.
	\end{itemize}
\end{definition}

Now, we can spell out the main property of disjunctions with stopping conditions.
\begin{theorem} \label{tw_alternatywy_z_warunkiem_stopu}
	Let $(M,T) \models \CT^-$ and let $(\alpha_i)_{i\leq c}, (\beta_i)_{i\leq c}$ be any coded sequences of sentences of $M$. Suppose that the least $k_0$ such that $(M,T) \models T \alpha_{k_0}$, is standard. Then 
	\begin{displaymath}
	(M,T) \models T \bigvee_{i=0}^{c ,\alpha}  \beta_i \equiv T \beta_{k_0}.
	\end{displaymath} 
\end{theorem} 

\begin{proof}
	We first show that 
	\begin{displaymath}
	(M,T) \models T \bigvee_{i=k_0}^{c ,\alpha}  \beta_i \equiv T \beta_{k_0}.
	\end{displaymath}
	Suppose that $(M,T) \models T \alpha_{k_0}$. Then by elementary propositional logic for any $\gamma$:
	\begin{displaymath}
	(M,T) \models (T\alpha_{k_0} \rightarrow T\beta_{k_0}) \wedge \Big((T\alpha_{k_0} \wedge T\beta_{k_0})  \vee (\neg T\alpha_{k_0} \wedge T\gamma) \Big)
	\end{displaymath}
	is equivalent to 
	\begin{displaymath}
	(M,T) \models T \beta_{k_0}.
	\end{displaymath}
	
	Then we prove by backwards (external) induction on $k$ that for any $k \leq k_0$ , 
	\begin{displaymath}
	(M,T) \models T \bigvee_{i=k}^{c ,\alpha} T \beta_i \equiv T \beta_{k_0}.
	\end{displaymath} 
	Suppose that this equivalence has already been proved for $k+1$. Then, since $k < k_0$ and we assumed that $k_0$ is minimal such that $(M,T) \models \alpha_{k_0}$, we know that $T \alpha_k$ does not hold and we have for an arbitrary $\gamma$:
	\begin{displaymath}
	(M,T) \models \Big[ (T\alpha_k \rightarrow  T\beta_k) \wedge \Big((T\alpha_k \wedge T\beta_k) \vee (\neg T\alpha_k \wedge T\gamma) \Big) \Big]\equiv T \gamma.
	\end{displaymath} 
	So, by induction hypothesis:
	\begin{displaymath}
	(M,T) \models T \bigvee_{i=k}^{c,\alpha} \beta_i \equiv T \bigvee_{i=k+1}^{c,\alpha} \beta_i \equiv T \beta_{k_0}.
	\end{displaymath}
	Which concludes the proof of the induction step.
\end{proof}

\section{Lachlan's Theorem}
In this section, we present a proof of Lachlan's theorem. We hope that our argument, although very similar to the original one, will be seen as less mysterious.

\begin{theorem}[Lachlan's Theorem] \label{tw_lachlana}
	Let $(M,T) \models \CT^-$. Then $M$ is recursively saturated. 
\end{theorem}

Let us describe the strategy of the proof. For a given coded and finitely satisfiable sequence of formulae $p = (\phi_i)_{i \in \omega}$, we will find a (nonstandard) formula $\gamma \in M$ such that 
\begin{itemize}
	\item $(M,T) \models \exists x T \gamma (\num{x})$;
	\item for all $i \in \omega$, $(M,T) \models \forall x \big(T\gamma(x) \rightarrow T \phi_i(x)\big).$
\end{itemize}

In other words, we will try to find a set of elements satisfying our type $p$ that is defined by a  nonstandard formula $\gamma$. In order to find $\gamma$, we will introduce a suitable notion of rank.

\begin{definition} \label{edf_p_rank}
	Let $(M,T) \models \CT^-$ and let $p = (\phi_i)_{i\in \omega}$ be any coded sequence of (possibly nonstandard) formulae. We define a $p$-\df{rank} of formulae $\phi \in \Form^{\leq 1}(M)$, $r_p(\phi)$ as follows:
	\begin{displaymath}
	r_p(\phi) = \left\{ \begin{array}{ll}
	- \infty, & \textnormal{if } (M,T) \models \neg \exists x \  T \phi(\num{x}) \vee \exists x \ \big( T\phi(\num{x}) \wedge \neg T \phi_0(\num{x}) \big); \\
	n, &  \textnormal{if } (M,T) \models \exists x \ T\phi(\num{x}) \textnormal{ and } n \in \omega \textnormal{ is the greatest such that } \\
	& (M,T) \models \forall x \ \big(T \phi(\num{x}) \rightarrow T \phi_i(\num{x})\big), \textnormal{for } i \leq n; \\
	\infty, & \textnormal{if }  (M,T) \models \exists x \ T \phi(\num{x}) \textnormal{ and } \\
	& \textnormal{for all $i\in\omega$, $(M,T) \models \forall x \big(T \phi(\num{x}) \rightarrow T \phi_i(\num{x})\big)$}.   
	\end{array} \right.
	\end{displaymath}
\end{definition}

We can say that $p$-rank of a formula measures  how close that formula gets to defining a set of elements satisfying the type $p$. Now, in order to prove Lachlan's theorem we will find a sequence $(\gamma_i)_{i < c}$ of formulae with $c$  nonstandard such that whenever $r_p(\gamma_a) \neq \infty$, then $r_p(\gamma_{a+1}) > r_p(\gamma_a).$ Then the theorem follows by a straightforward application of Lemma \ref{lem_dobre_porzadki_w_M} for $f(x) = r_p(\gamma_x)$.

\begin{lemma}[Rank Lemma] \label{lem_lemat_o_randze_dla p_rangi}
	Let $(M,T) \models \CT^-$. Then there exists a coded sequence of formulae $(\gamma_i)_{i<c}$ of nonstandard length such that for all $a <c $ either $r_p(\gamma_a) = \infty$ or
	\begin{displaymath}
	r_p(\gamma_{a+1}) > r_p(\gamma_a).
	\end{displaymath}
\end{lemma}

\begin{proof}
	Fix $(M,T) \models \CT^-$ and $p = (\phi_i)_{i<c}$, a coded sequence of arithmetical formulae such that for any $k \in \omega$,
	\begin{displaymath}
	(M,T) \models \exists x \ T \bigwedge_{i \leq k} \phi_i(\num{x}) .
	\end{displaymath}
	Without loss of generality we can additionally assume that for any $i<j\in\omega$, 
	\begin{displaymath}
	(M,T) \models T  \forall x \ \Big(  \phi_j(x) \rightarrow \phi_i(x) \Big).
	\end{displaymath}
	We will define the sequence $(\gamma_i)$ using disjunctions with a stopping condition.	First, notice that for a given  formula $\psi$, we can express that it has rank smaller than $n$.\footnote{This is slightly different than in the informal discussion in the previous section, where for clarity's sake $\alpha_n[\psi]$ was taken to mean that $\psi$ has \emph{greatest} possible rank smaller than $n$.} Let:
	\begin{eqnarray*}
	\alpha_0[\psi] & : = &  \neg \exists x \   \psi(x) \vee \exists x \ \big( \psi(x) \wedge \neg  \phi_0(x) \big) \\	
	\alpha_n[\psi] & : = & \exists x [\psi(x) \wedge \neg \phi_{n}(x)], \textnormal{ for $n>0$}.
	\end{eqnarray*}
	and (to keep our notation consistent)
	\begin{displaymath}
	\beta_n(x) := \phi_n(x).
	\end{displaymath} 
	Then,  for all $n\in \omega$, we have $r_p(\beta_n)\geq n$ and  
	\begin{displaymath}
	(M,T) \models T \alpha_n[\psi] \textnormal{ implies } r_p(\psi) < n,
	\end{displaymath}
	
		Now, we are in position to define a coded sequence $(\gamma_i)_{i<d}$ of formulae of nonstandard length which satisfies the conditions of the lemma. 
	
	Fix any nonstard $d$ and let
	\begin{eqnarray*}
		\gamma_0(x) & : = & (x=x) \\
		\gamma_{j+1}(x) & : = & \bigvee_{i=0}^{d, \alpha[\gamma_j]} \beta_i(x).
	\end{eqnarray*}

Let us check that $\gamma_i$ indeed satisfies the conditions of the lemma. Suppose that 
\begin{displaymath}
r_p(\gamma_a) \neq \infty.
\end{displaymath}
If $r_p(\gamma_a) = -\infty$, then $(M,T) \models T \alpha_0[\gamma_a]$. If $r_p(\gamma_a) = n$ for some  $n \in \omega$, then
\begin{displaymath}
(M,T) \models T \alpha_{n+1}[\gamma_a].
\end{displaymath}
Let $k$ be the least number such that
\begin{displaymath}
(M,T) \models T \alpha_k[\gamma_a].
\end{displaymath}
 Then by Theorem \ref{tw_alternatywy_z_warunkiem_stopu}, we see that
\begin{displaymath}
(M,T) \models \forall x \Big( T \gamma_{a+1}(\num{x}) \equiv  T \bigvee_{i=0}^{d, \alpha[\gamma_a]} \beta_i(\num{x}) \equiv T \beta_{k} (\num{x}) \Big).
\end{displaymath}
As we have already observed, $r_p(\beta_{k}) \geq  k > r_p(\gamma_a)$, so the sequence $(\gamma_i)_{i<d}$ satisfies the claim of the lemma.
\end{proof}

Now, Lachlan's Theorem follows immediately from Lemma  \ref{lem_dobre_porzadki_w_M} and Lemma \ref{lem_lemat_o_randze_dla p_rangi}.

\begin{remark}
	Notice that we can obtain a number of stronger results by inspection of the above proof. The modifications go in different directions and are sometimes mutually exclusive. Let us now list them.
	\begin{enumerate}
		\item Actually, the proof shows that any type coded in a model $(M,T) \models \CT^-$ is satisfied in that model. 
		\item Even stronger, the proof shows that if $(\phi_i)$ is a coded sequence of (possibly nonstandard) formulae such that for any $n$, there exists $x \in M$ for which $T\phi_i(\num{x})$ holds for $i \leq n$, then there exists $x \in M$ such that $T \phi_n(\num{x})$ holds for all $n \in \omega$. This result has been first formulated \cite{smith}, where it is attributed to an anonymous referee.
		\item In the proof, we do not use the full strength of $\PA$. Actually, $ \IDelta_0 + \exp$ is enough. We only need to apply syntactic operations to arbitrary formulae and to make iterations of these operations of some nonstandard length. 
		\item The proof actually works for $\CT^- \res I$ for a nonstandard cut $I$. Indeed, under such assumptions, we only need to additionally ensure that we take disjunctions with stopping conditions small enough so that they belong to the cut $I$.
		\item Actually, we can combine some of the above modifications: if $M \models \IDelta_0 + \exp$ expands to a model of $\CT^- \res I$ for nonstandard $I$, then $M$ realises all coded types. 
		\item The proof works with the binary satisfaction predicate (operating on formulae and valuations) in place of the truth predicate.
		\item We could define a natural analogue of $\CT^- \res I$ for a satisfaction predicate, a predicate which satisfies compositional conditions for arbitrary valuations and formulae with depth from a nonstandard cut. If a model of $\IDelta_0 + \exp$ expands to a model of such a theory, then it realises all coded arithmetical types.  
	\end{enumerate} 
\end{remark}


\section{Definability of partial inductive truth predicates}

In this section, we will present a refinement of Lachlan's Theorem which is also a strengthening of Smith's Theorem that in every model $(M,T) \models \CT^-$ there is an undefinable class (\cite{smith}, Theorem 2.10). A related result was proved in \cite{lelykwcislomodels}, Theorem 4.1. 

\begin{theorem} \label{th_ctminus_definiuje_utb}
	Let $(M,T) \models \CT^-$. Then there exists $T' \subset M$ that is definable in $(M,T)$, such that
	\begin{displaymath}
	(M,T') \models \CT \res c
	\end{displaymath}
	for a nonstandard $c \in M$.
\end{theorem}

The proof will closely follow our argument from the previous section: we will define a suitable notion of rank and demonstrate that there is a coded sequence of formulae whose rank is increasing. 

We will try to find a (nostandard) formula $\gamma$ such that $T'$ is defined as $\gamma(M) : = \set{x \in M}{(M,T) \models T \gamma(\num{x})}$. Our rank will measure how close a given formula $\gamma$ gets to  defining a truth predicate satisfying $\CT \res c$. Such a rank can be found thanks to the following proposition.

\begin{proposition} \label{stw_ctc_wynika_z_utb}
	Let $M \models \PA$. Suppose that $(M,T')$ satisfies full induction in the extended language, structural regularity property $\SRP$, and the following scheme of uniform Tarski's biconditionals:
	\begin{displaymath}
	 \forall \bar{s} \in \ClTermSeq_{\PA} \Big( (T'(\num{\phi(\bar{s})})) \equiv  \phi(\bar{\val{s}})\Big)
	\end{displaymath}
	for all (standard) arithmetical formulae $\phi$. Then there exists a nonstandard $c \in M$ and $T'' \subset T'$ such that
	\begin{displaymath}
	(M,T'') \models \CT \res c.
	\end{displaymath}
\end{proposition} 
The proposition can be proved using an easy overspill argument. Notice that if a sentence $\phi$ has standard syntactic depth $n \in \omega$, then there exists a standard sentence $\psi \simeq \phi$, so $\SRP$ allows us to conclude that the truth predicate behaves compositionally on all sentences of standard complexity.

Let $(\ind_i(P))$ be a primitive recursive enumeration of all instances of the induction scheme with one extra second-order variable $P$. Then, slightly abusing the notation, we write for a (possibly nonstandard) formula $\psi$:
\begin{displaymath}
\ind_i(\psi)
\end{displaymath}
meaning the $i$-th instance of the induction scheme with the formula $\psi$ substituted for the variable $P$. 

Let $(\phi_i)$ be a primitive recursive enumeration of arithmetical formulae. Now, we are ready to define a suitable notion of rank:

\begin{definition} \label{def_ranga_dla_utb}
	Let $(M,T) \models \CT^-$ and let $\gamma \in \Form^{\leq 1}(x)$. We define a rank of the formula $\gamma$, $r(\gamma)$ as follows:
	\begin{displaymath}
	r(\gamma) = \left\{ \begin{array}{ll}
	- \infty, & \textnormal{if } (M,T) \models \neg \exists x T \gamma(\underline{x}) \vee \neg T \ind_0(\gamma) \vee  \\
	& \exists \phi, \psi \in \Sent_{\PA} \ \phi \simeq \psi \wedge \neg T \Big(\gamma(\num{\phi}) \equiv \gamma(\num{\psi}) \Big) \vee \\
	&  \neg T \ \forall \bar{s} \in \ClTermSeq_{\PA} \Big(  (\gamma(\num{\phi_0(\bar{s})})) \equiv  \phi_0(\bar{\val{s}})\Big). \\
	n, &  \textnormal{if } (M,T) \models \exists x \ T\gamma(\underline{x}) \textnormal{ and } n \in \omega \textnormal{ is the greatest such that } \\
	& (M,T) \models T \bigwedge_{i \leq n} \Big[\ind_i(\gamma)  \wedge  \forall \bar{s} \in \ClTermSeq_{\PA} \Big( (\gamma(\num{\phi_i(\bar{s})})) \equiv  \phi_i(\bar{\val{s}})\Big)\Big] \\
	\infty, & \textnormal{if }  (M,T) \models \exists x T \phi(\num{x}) \textnormal{ and for all $i\in\omega$,} \\
	& (M,T) \models T \Big[\ind_i(\gamma)  \wedge \forall \bar{s} \in \ClTermSeq_{\PA} \Big( (\gamma(\num{\phi_i(\bar{s})})) \equiv  \phi_i(\bar{\val{s}})\Big) \Big]. 
	\end{array} \right.
	\end{displaymath}
\end{definition}

To find the required $\gamma$ with $r(\gamma) = \infty$, we will  use Lemma \ref{lem_lemat_o_randze_dla p_rangi}. 

As in the previous section, notice that we can express that $\psi$ has rank smaller than $n$. Let
\begin{eqnarray*}
	\alpha_0[\psi] & : = & \neg \exists x \psi(x) \vee \neg \ind_{0}(\psi) \vee \\
	& & \exists \bar{s} \in \ClTermSeq_{\PA} \neg \Big(\psi(\phi_{0}(\bar{s})) \equiv \phi_{0}(\bar{\val{s}}) \Big) \vee \\
	&  & \exists \phi, \phi' \in \Sent_{\PA} \ \Big( \phi \simeq \phi' \wedge \neg \big( \psi(\num{\phi}) \equiv \psi(\num{\phi'})  \big) \Big ). \\
	\alpha_n[\psi] & : = & \neg \ind_{n}(\psi) \vee \exists \bar{s} \in \ClTermSeq_{\PA} \neg \Big(\psi(\phi_{n}(\bar{s})) \equiv \phi_{n}(\bar{\val{s}}) \Big) \textnormal{ for } n>0.
\end{eqnarray*}

We can also readily find formulae, whose rank equals at least $n$. Let 
\begin{displaymath}
\beta_n(x) = \bigvee_{i=0}^n [\exists \bar{s} \ \ x \simeq \phi_i(\bar{s}) \wedge \phi_i(\bar{\val{s}})].
\end{displaymath}

As in the previous section, we define a coded sequence of formulae $\gamma_i$:
\begin{eqnarray*}
\gamma_0(x) & : = & (x=x) \\
\gamma_{j+1}(x) & : = & \bigvee_{i=0}^{d, \alpha[\gamma_j]} \beta_i(x).
\end{eqnarray*} 

Now, we are in position to formulate and prove an analogue of Lemma \ref{lem_lemat_o_randze_dla p_rangi}.

\begin{lemma} \label{lem_lemat_o_randze_dla_utb}
	Let $(M,T) \models \CT^-$. Then for any $a \in M$ either $r(\gamma_a) = \infty$ or 
	\begin{displaymath}
	r(\gamma_{a+1}) > r(\gamma_a).
	\end{displaymath} 
\end{lemma} 
\begin{proof}
	Suppose that $r( \gamma_a) \neq \infty$. This means that $r (\gamma_a) = - \infty$ or $r (\gamma_a ) = n \in \omega$. Then we have 
	\begin{displaymath}
	(M,T) \models T \alpha_k[\gamma_a]
	\end{displaymath}
	for $k = 0 $ or $k= n+1$, respectively, and $k$ is the least such number. Therefore by Theorem \ref{tw_alternatywy_z_warunkiem_stopu}, we have:
\begin{displaymath}
(M,T) \models \forall x \Big( T \gamma_{a+1}(\num{x}) \equiv T \bigvee_{i=0}^{d,\alpha} \beta_i(\num{x}) \equiv T \beta_k (\num{x}) \Big) .
\end{displaymath}
But, by our construction, $r (\beta_k) \geq n+1 > r (\gamma_a).$
\end{proof}

Theorem \ref{th_ctminus_definiuje_utb} follows immediately by Proposition \ref{stw_ctc_wynika_z_utb}, Lemma \ref{lem_dobre_porzadki_w_M} for $f(x) = r(\gamma_x)$ and Lemma \ref{lem_lemat_o_randze_dla_utb}.

\section{Non-extendable partial truth predicates}

In this section, we apply disjunctions with stopping condition to study extensions of models of $\CT^-$. We are dealing with the following question. Suppose that $M \models \PA$, $I \subset M$ is a nonstandard cut, and $(M,T) \models \CT^- \res I$.  Is there a $T' \supset T$ such that $(M,T') \models \CT^-$?

The above question asks about possible obstructions to the existence of a fully compositional truth predicate. The most classical result in this vein is Lachlan's Theorem which amounts to saying that in some models $M \models \PA$, the natural truth predicate defined on formulae of standard complexity (the unique smallest predicate satisfying $\CT^- \res \omega$) cannot be extended to a full truth predicate. 

As we have already remarked,  if $M \models \PA$ is not a recursively saturated model, then one cannot find a truth predicate $T$ satisfying $\CT^- \res I$ for a nonstandard cut $I$. The proof of Lachlan's Theorem applies with some additional care paid to the choice of parametres so that the depths all relevant formulae are in the cut $I$. Now, our question in this section asks whether once a truth predicate is already defined on a \emph{nonstandard} cut of formulae, there can be any further obstructions to extending it to the whole model.

This question may be also viewed from a slightly different angle. Smith has proved in (\cite{smith}, Theorem 4.3) that there exists a model $(M,T) \models \CT^-$ such that it cannot be end-extended to another model of $\CT^-$. In the proof of Smith's theorem one shows that such an extension cannot be found if a nonstandard formula $\phi$ defines a surjection from a cut $J$ to the whole model (i.e., the formula $T\phi(\num{x}, \num{y})$ is functional in $x$ and defines that surjection).

Now, we can ask the question, whether this is essentially the only possible obstruction. We asked this question trying to show that if $(M,T) \models \CT^-$ and $T$ believes all the instances of induction to be true, then it has an end extension. Notice that such a truth predicate cannot display the pathology used by Smith. Moreover, one can show that in such case, under an additional assumption that $(M,T) \models \SRP$,  there exists a proper end extension $(M,T) \subset_e (N,T')$ such that $M \preceq_e N$ and  $(N,T') \models \CT^- \res M$. This leads us to the following question about extensions of $\CT^-$: let $(M,T) \models \CT^- + \SRP$. Suppose that $M \preceq_e N$ is an elementary end extension. Let $T \subset T' \subset N$ be a truth predicate satisfying $\CT^- \res M$. In particular, we know that $(M,T)$ is free of pathologies employed by Smith. Does there exist $T'' \supset T'$ such that $(N,T'') \models \CT^-$?

We answer both questions in the negative. We will give a proof for a general cut satisfying some additional conditions. It is however known that such cuts may be even required to be elementary initial segments which are recursively saturated models of $\PA$. Every recursively saturated model of PA has elementary cuts that satisfy the condition.

\begin{theorem} \label{tw_ctminus_co_sie_nie_rozszerza}
	Let $M \models \PA$ be a countable recursively saturated model and let  $I = \set{x \in M }{\exists n \in \omega \ \  x < a_n}$ for some increasing coded sequence $a \in M$ such that the difference $a_{n+1} - a_n$  is nonstandard for any $n$. Then there exists $T \subset M$ such that $(M,T) \models \CT^- \res I$, but there is no $T' \supset T$ such that $(M,T') \models \CT^-$. 
\end{theorem}

A slight modification of the proof yields the following result:

\begin{theorem} \label{tw_ctminus_co_sie_nie_rozszerza_z_podmodelu}
		Let $M \preceq_e N$ be a countable recursively saturated models of $\PA$. Suppose that  $M = \set{x \in N }{\exists n \in \omega \ \  x < a_n}$ for some increasing coded sequence $a \in M$ such that the difference $a_{n+1} - a_n$  is nonstandard for any $n$. 
	Then there exists $T \subset N$ such that $(N,T) \models \CT^- \res M$ and  $(M,T\cap M) \models \CT^-$, but there is no $T' \supset T$ such that $(N,T')  \models \CT^-$.
\end{theorem}

The difference between this theorem and the previous one is that now we explicitly require that $(M,T \cap M) \models \CT^-$. This means in particular that any existential formula from $M$ which is rendered true by the predicate $T$ must have a witness already in $M$. Note that considering the special case of standard formulae with nonstandard numerals denoting elements from $M$, we can conclude that $M$ is an elementary submodel of $N$. Since the proof of Theorem \ref{tw_ctminus_co_sie_nie_rozszerza_z_podmodelu} is a modification of the proof of Theorem \ref{tw_ctminus_co_sie_nie_rozszerza}, we will only briefly comment on what needs to be changed.

 Incidentally, Theorem \ref{tw_ctminus_co_sie_nie_rozszerza} holds for an arbitrary cut $I \subsetneq M$, but for rather uninteresting reasons. The way we defined it, if $I \subset J$ are two cuts and $(M,T) \models \CT^- \res J$, then also $(M,T) \models \CT^- \res I$. Therefore, we could take an arbitrary cut $I$, find a bigger cut $J$ with a coded cofinal $\omega$-sequence, and apply Theorem \ref{tw_ctminus_co_sie_nie_rozszerza} in its current version.\footnote{We are grateful to Jim Schmerl for this remark.}

Regarding Theorem \ref{tw_ctminus_co_sie_nie_rozszerza_z_podmodelu}, notice that if $N \models \PA$ is recursively saturated, then arbitrarily high we can find cuts satisfying the assumptions of the theorem, i.e. cuts $M$ such that $M \preceq_e N$ and $M$ has a cofinal sequence of length $\omega$ coded in $N$. Indeed, take any $a \in N$ and construct a series of elements $(a_n)_{n \in \omega}$ such that $a_0 = a$ and for any $n$, the element $a_{n+1}$ dominates all functions arithmetically definable with parametres less or equal to $a_n$. That such an element exists follows from recursive saturation. Then, $M = \set{x \in N}{\exists n \in \omega \ \ x<a_n}$ is an elementary submodel of $N$. Moreover, it can be easily verified that $M$ has to be recursively saturated itself.

The proof of Theorem \ref{tw_ctminus_co_sie_nie_rozszerza} relies on the following lemma. In the lemma we will use certain formulas $\eta_b$. For  $b \in M$, let  $\eta_b$ be
\begin{displaymath}
\exists x_0 \ldots \exists x_b v=v \wedge \bigwedge_{i=0}^b x_i = x_i.
\end{displaymath}
 Notice that $\sd(\eta_b) = 2b +2$, which will be handy in the proof of the lemma. We also introduce the following notation: if $M \models \PA$, $I \subset  M$ is an initial segment, and $T \subset \Sent_{\PA}(M)$, then by $T \res I$ we mean $\set{x \in \Sent_{\PA}(M)}{T(x) \wedge \sd(x) \in I}$. 

\begin{lemma} \label{lem_rozszerzenia_ctminus_o_elastyczne_formuly}
	Let $(M,T,J) \models \CT^- \res J$ be countable and recursively saturated as a model in the expanded language. Let $A \subset M$ be any set such that $(M,T,A,J)$ is recursively saturated. Then, for any $b \notin J$, there exists $T' \supset T \res J$ such that $(M,T') \models \CT^-$ and the formula $T'\eta_b(\num{v})$ defines $A$.
\end{lemma}
 
Since the proof of the lemma is a modification  of the Enayat--Visser conservativity proof for $\CT^-$, we move it to the appendix.

\begin{proof}[Proof of Theorem \ref{tw_ctminus_co_sie_nie_rozszerza}]
	Let $M$ be a countable recursively saturated model of $\PA$. Let $a \in M$, let $I = \set{x \in M}{\exists n \in \omega \ x< a_n}$, and let $(b_n)_{n < \omega}$ be a decreasing sequence  such that 
	\begin{displaymath}
	\set{x  \in M}{\forall n \ x<b_n} = \omega.
	\end{displaymath}
	
	We construct the predicate $T$ by recursion. Let $T'_0$ be any truth predicate such that $(M,T'_0) \models \CT^-$ is recursively saturated and
	\begin{displaymath}
	(M,T'_0) \models \forall x \ \ \big( T'_0 \eta_{a_0}(\num{x}) \equiv x=b_0 \big)
	\end{displaymath}
	Let $T_0$ be $T'_0$ restricted to formulae in $J_0$ where $a_0 \in J_0$, $a_1  \notin J_0$, and such that $(M,T_0, J_0)$ is recursively saturated. 
	
	Suppose that $T_n$ is a truth predicate such that $(M,T_n,J_n) \models \CT^- \res J_n$ where $a_n \in J_n, a_{n+1} \notin J_n$, $(M,T_n, J_n)$ is recursively saturated, and for all $i \leq n$,
	\begin{displaymath}
	(M,T_n) \models \forall x \ \ \big( T_{n} \eta_{a_{i}}(\num{x}) \equiv x=b_i \big).
	\end{displaymath}
	
	Using Lemma \ref{lem_rozszerzenia_ctminus_o_elastyczne_formuly}, we find $T_n \subset T_{n+1}' \subset M$ such that $(M,T_{n+1}') \models \CT^-$ is a recursively saturated model such that
	\begin{displaymath}
	(M,T'_{n+1}) \models \forall x \ \ \big( T'_{n+1} \eta_{a_{n+1}}(\num{x}) \equiv x=b_{n+1} \big).
	\end{displaymath}
	We set $T_{n+1} = T'_{n+1} \res J_{n+1}$, where $a_{n+1} \in J_{n+1}, a_{n+2} \notin J_{n+1}$ and $(M,T_{n+1},J_{n+1})$ is recursively saturated. One readily checks that $T_{n+1}$ satisfies our inductive conditions.

	Finally, we set $T = \bigcup_{i \in \omega} T_i$. Then 
	\begin{displaymath}
	(M,T) \models \CT^- \res I
	\end{displaymath}
	and the formulae $T \eta_{a_n}(\num{x})$ define the elements $b_n$.

	Now we can use the machinery of disjunctions with stopping conditions to show that $T$ cannot be extended to $T'$ such that $(M,T') \models \CT^-$. Suppose towards contradiction that such a $T'$ can be found. Again, we introduce a suitable notion of rank. For $\phi \in \Form^{\leq 1}_{\PA}(M)$, let 
	 \begin{displaymath}
	 r(\phi) = \left\{ \begin{array}{ll}
	 - \infty, & \textnormal{if } (M,T) \models \neg \exists x \ \  T \phi(x) \vee \forall x \big(T \phi(\num{x}) \rightarrow x>b_0\big). \\
	 n, &  \textnormal{if } (M,T) \models \exists x \ \ T \phi(x) \textnormal{ and } n \in \omega \textnormal{ is the greatest such that } \\
	 & (M,T) \models \forall x \Big(T \phi(\num{x}) \rightarrow x \geq n \wedge x \leq b_n \Big) \\
	 \infty, & \textnormal{if }  (M,T) \models \exists x \ \ T \phi(x) \textnormal{ and } \\
	 & \textnormal{for all } n \in \omega (M,T) \models \forall x \Big( T \phi(\num{x}) \rightarrow x \geq n \wedge x \leq  b_n \Big) \\
	 \end{array}  \right.
	 \end{displaymath}
	 
	Since $(b_n)$ is downwards cofinal in $M \setminus \omega$, one can readily see that there are no formulae of rank $\infty$, because an element defined with such a formula necessarily would have to be between $\omega$ and all elements $b_n$. Notice that for any formula $\phi$, we can in fact find a \emph{coded} sequence of sentences $\alpha_i[\phi]$ such that 
	\begin{displaymath}
	(M,T) \models T \alpha_i[\phi] \textnormal{ iff } r(\phi) < n.
	\end{displaymath}
	Namely, we set:
	\begin{eqnarray*}
	\alpha_0[\phi] & := & \neg \exists x \    \phi(x) \vee \forall x, y \big( \phi(x) \wedge \eta_{a_0}(y) \rightarrow x>y\big) \\
	\alpha_n[\phi] & := & \exists x,y \big( \phi(x) \wedge   \eta_{a_n}(y) 
	\wedge (x < n \vee x > y) \big).
	\end{eqnarray*}

	Using Lemma \ref{lem_dobre_porzadki_w_M}, it is enough to find a coded sequence of sentences growing in the rank. Fix any $c$ smaller than the length of $a$ as a sequence (where $a$ is the coded sequence that we have fixed in the construction of our predicate $T$) and let
	\begin{eqnarray*}
	\gamma_0(x) & : = & (x=x) \\
	\gamma_{j+1}(x) & : = & \bigvee_{i=0}^{c, \alpha[\gamma_j]} \eta_{a_{i+1}}(x).
	\end{eqnarray*}
	We claim that for all $d < a$ either $r(\gamma_d) = \infty$ or $r(\gamma_{d+1}) > r(\gamma_d)$.
	
	Fix any $d$ and suppose that $r(\gamma_d) = - \infty$ or $r(\gamma_d) = n \in \omega$. Then by Theorem \ref{tw_alternatywy_z_warunkiem_stopu}
	\begin{displaymath}
	(M,T') \models \forall x \big( T \gamma_{d+1}(\num{x}) \equiv T \eta_{a_{k}}(\num{x}) \big)
	\end{displaymath}
	where $k = 0$ if $r(\gamma_d) = - \infty$ or $k= n+1$ if $r(\gamma_c) = n \in \omega$. The rank of the formula $\eta_{a_{k}}$ is greater than $r(\gamma_d)$, since $\eta_{a_k}$ defines the element $b_k$ and the sequence $(b_n)$ is decreasing. Now, as in proofs of Theorems \ref{tw_lachlana} and \ref{th_ctminus_definiuje_utb}, Lemma \ref{lem_dobre_porzadki_w_M} for $f(x) = r(\gamma_x)$ would imply that there exists a formula $\gamma$ with rank equal to $\infty$, and, as we have already noted, such a formula cannot exist. 
\end{proof}

Now let us comment on the modifications to the above construction needed in order to prove Theorem \ref{tw_ctminus_co_sie_nie_rozszerza_z_podmodelu}. The crucial problem is that the constructed truth predicate restricted to $I$ now needs to be a model of $\CT^-$ itself. In order to achieve this, we can take every $J_n$ to be an elementary submodel of $M$ such that $(M,J_n)$ is recursively saturated. We additionally require that each $T_n$ has the additional property that $(J_n,T_n \cap J_n) \models \CT^-$. This can be proved similarly to Lemma \ref{lem_rozszerzenia_ctminus_o_elastyczne_formuly}, but we skip the unenlightening details. 

\begin{remark} \label{rem_dodanie_SRP}
	Theorems \ref{tw_ctminus_co_sie_nie_rozszerza} and \ref{tw_ctminus_co_sie_nie_rozszerza_z_podmodelu} remain true if we additionally require that the  truth predicate $T$ satisfies the structural regularity property $\SRP$. The proof is entirely analogous, since one can show a modified version of Lemma \ref{lem_rozszerzenia_ctminus_o_elastyczne_formuly} in which both the initial truth predicate $T$ and the constructed truth predicate $T'$ are required to satisfy $\SRP$.
\end{remark}

\section{Appendix}

In this section, we prove Lemma \ref{lem_rozszerzenia_ctminus_o_elastyczne_formuly}. Let us restate it, for the convenience of the reader:
\begin{lemma*}
	Let $(M,T,J) \models \CT^- \res J$ be countable and recursively saturated as a model in the expanded language. Let $A \subset M$ be any set such that $(M,T,A,J)$ is recursively saturated. Then, for any $b \notin J$, there exists $T' \supset T \res J$ such that $(M,T') \models \CT^-$ and the formula $T'\eta_b(\num{v})$ defines $A$.
\end{lemma*}
 Its proof is a  modification of the construction by Enayat and Visser from \cite{enayatvisser}. 

The lemma is a strengthening of a result by Smith (\cite{smith}, Theorem 3.3) who showed that any $A \subseteq M$ such that $(M,A)$ is recursively saturated may be defined with a nonstandard formula. In the above Lemma, we additionally require that we may arbitrarily fix this truth predicate on any given cut.

\begin{proof}[Proof of Lemma \ref{lem_rozszerzenia_ctminus_o_elastyczne_formuly}]
Recall that $\eta_b$ was defined as:
\begin{displaymath}
\exists x_0 \ldots \exists x_b \ v=v \wedge \bigwedge_{i=0}^b x_i = x_i.
\end{displaymath}

	Fix $(M,T, J, A)$ as in the assumptions of the lemma. We first show that there exists an extension 
	\begin{displaymath}
	(M,T,J,A) \preceq (M',T',J',A')
	\end{displaymath}
	 and $T''\subset M'$ such that
	\begin{itemize}
		\item $(M',T'') \models \CT^-$;
		\item $(M',T'',A') \models \forall x \ \ x\in A' \equiv T'' \eta_b(\num{x})$;
		\item $T' \res J' \subset T''$. 
	\end{itemize}
	By resplendency of $(M,T,A)$, this will conclude our proof. 
	
	In order to construct $(M',T',J',A',T'')$, we build an auxiliary chain of models: $(M_n,T_n,J_n,A_n,S_n)$ of length $\omega$ such that $T_n$ and $S_n$ are binary relations (we replace truth predicates with satisfaction predicates), $J_n$ is a cut, and $A_n\subseteq M_n$.	We assume for convenience that $T \res J = T$, i.e., $T$ is only defined on formulae whose depth is in $J$.
	
	We define $A_0$ as $A$, $M_0$ as $M$, $J_0$ as $J$. $S_0$ is the empty set, and $T_0$ is a partial satisfaction predicate defined so that $T_0(\phi,\alpha)$ holds for $\phi \in \Form_{\PA}(M_0)$, $\alpha \in \Val(\phi)$ if $T(\phi[\alpha])$ holds, where $\phi[\alpha]$ is obtained from $\phi$ by substituting $\num{\alpha(v)}$ for every $v \in \FV(\phi)$; i.e.~we take a free variable $v$ in $\phi$, see what its value is under $\alpha$, we take the canonical numeral denoting this value, and we substitute it for $v$. Similarly, if $t \in \Term_{\PA}$, and $\alpha$ is a valuation defined on its free variables, then by $t[\alpha]$ we mean the value of the term $t$ with numerals $\num{\alpha(v)}$ substituted for free variables $v$ in the term $t$. If $\alpha, \beta$ are valuations and $v$ is a variable, we denote by $\alpha \sim_v \beta$ that $\alpha$ and $\beta$ are identical, possibly except for the value on the variable $v$ (which is in particular allowed not to be an element of  the domain of $\beta$).
	
	We inductively construct a chain of countable models $(M_n,T_n, J_n, A_n, S_n)$ of length $\omega$. Suppose that we have already defined the $n$-th model in the chain. Then we define $(M_{n+1}, T_{n+1}, J_{n+1}, A_{n+1}, S_{n+1})$ as any model of the theory $\Theta_{n}$ with the following axioms:
	\begin{itemize}
		\item The elementary diagram $\ElDiag(M_n,T_n, J_n, A_n)$ (with symbols $A_{n}, T_{n}, J_n$ replaced with $A_{n+1}$, $T_{n+1}$, $J_{n+1}$, respectively). 
		\item The compositionality scheme $\Comp_n(\phi)$, for $\phi \in \Form_{\PA}(M_n)$, to be defined later.
		\item The regularity axiom I: $\forall \phi \in \Form_{\PA}, \alpha \in \Val (\phi) \ \ S_{n+1}(\phi,\alpha) \equiv S_{n+1}(\phi[\alpha],\emptyset)$.
		\item The regularity axiom II: $\forall \phi \in \Form_{\PA} \forall \bar{s}, \bar{t} \in \ClTermSeq_{\PA} \ \ \bar{\val{s}} = \bar{\val{t}} \rightarrow S_{n+1} (\phi(\bar{s}), \emptyset) \equiv S_{n+1} (\phi(\bar{t}), \emptyset)$.
		\item $\forall \phi \in \Form_{\PA} \forall \alpha \in \Val(\phi) \ \ T_{n+1}(\phi, \alpha) \rightarrow S_{n+1}(\phi, \alpha)$.
		\item $\forall x \ \ x \in A_{n+1} \equiv S_{n+1}(\eta_b(\num{x}),\emptyset)$.
		\item An additional preservation condition for $n>0$: $S_{n+1}(\phi,\alpha)$ for all $\phi \in \Form_{\PA}(M_{n-1}), \alpha \in \Val(\phi) \in M_n$ such that $S_n(\phi,\alpha)$ holds. (By convention, we set $M_{-1} = \emptyset$.)
	\end{itemize}

	Finally, an instance of the compositionality scheme $\Comp_n(\phi)$ is defined as the conjunction of the following axioms:
	\begin{itemize}
		\item $\forall s,t \in \Term_{\PA} \forall \alpha \in \Val(\phi) \ \Big( \phi = (s=t) \rightarrow S_{n+1}(\phi, \alpha) \equiv s[\alpha]=t[\alpha]$ \Big).
		\item $\forall \psi \in \Form_{\PA} \forall \alpha \in \Val(\phi) \ \ \Big( \phi = \neg \psi \rightarrow S_{n+1}(\phi,\alpha) \equiv \neg S_{n+1} (\psi,\alpha) \Big)$. 
		\item $\forall \psi, \eta \in \Form_{\PA} \forall \alpha \in \Val(\phi) \ \ \Big( \phi = (\psi \vee \eta) \rightarrow S_{n+1}(\phi,\alpha) \equiv S_{n+1}(\psi,\alpha) \vee S_{n+1}(\eta,\alpha) \Big)$.
		\item $\forall v \in \Var, \psi \in \Form_{\PA} \forall \alpha \in \Val(\phi) \ \ \Big( \phi = (\exists v \psi) \rightarrow S_{n+1}(\phi, \alpha) \equiv \exists \alpha' \sim_{v} \alpha \ \  S_{n+1}(\psi, \alpha') \Big)$.
	\end{itemize}

Let us assume that $\Theta_n$ is consistent. We will actually prove it later. Assuming that the construction works (i.e., all the models $(M_n,T_n,J_n,A_n,S_n)$ exist), we define:
\begin{itemize}
\item $M' = \bigcup M_n$.
\item $T' = \set{\phi \in \Sent_{\PA}(M')}{(\phi,\emptyset) \in \bigcup T_n }.$
\item $J' = \bigcup J_n$.
\item $A' = \bigcup A_n$.
\item $T'' = \set{\phi \in \Sent_{\PA}(M')}{\exists n \ \phi \in \Sent_{\PA} (M_n) \wedge (\phi, \emptyset) \in S_{n+1}}.$
\end{itemize} 
We claim that $(M',T',J',A',T'')$ satisfies the conditions listed at the beginning of our proof. Let us check it. 

The elementarity of the extension   $(M,T,J,A) \preceq (M',T',J',A')$ follows from the fact that every extension in the constructed chain was elementary in this restricted language. The containment $T' \subseteq T''$ also follows from the fact that the containment holds at every step of our construction. 

Let us now observe that if $\phi \in M_n, \alpha \in M_{n+1}$, and $(\phi,\alpha) \notin S_{n+1}$, then $(\phi, \alpha) \notin S_l$ for $l \geq n+1$. Indeed, if $(\phi, \alpha) \notin S_{n+1}$, then by compositional conditions $(\neg \phi, \alpha) \in S_{n+1}$ and, consequently $(\neg \phi, \alpha) \in S_l$ which, again by compositional axioms, implies $(\phi, \alpha) \notin S_l$. The equivalence  
\begin{displaymath}
 \forall x \ \ A_n(x) \equiv  S_n(\eta_b(\num{x}), \emptyset)
\end{displaymath}
also holds for every $n>0$. The predicates $A_n$ extend each other elementarily. This guarantees that $\eta_b$ defines the set $A'$ in the model $(M',T'')$. Now it suffices to check that $(M',T'') \models \CT^-$. 

Let us fix any $\phi \in M'$. We prove compositionality by cases considering various possible syntactic forms of $\phi$. Let us consider for example the case when $\phi = \exists v \psi(v)$. (We omit the other cases which follow by similar, simpler arguments.) Fix the least $n$ such that $\phi \in M_n$. Suppose that $\phi \in T''$.  By definition, this means that $(\exists v \psi, \emptyset) \in S_{n+1}$. By compositional conditions, there is a valuation $\alpha$ defined on the variable $v$ such that $(\psi, \alpha) \in S_{n+1}$ and, by the regularity axiom I, $(\psi[\alpha], \emptyset) \in S_{n+1}$ as well. Fix any variable $w$ which does not occur in $\psi$ such that it minimises $k$ for which $\psi' := \psi[w/v] \in M_k$. Let $\beta$ be a valuation defined only on $w$ such that $\beta(w) = \alpha(v)$. Then, by the regularity axiom I, $(\psi', \beta) \in S_{n+1}$, which implies $(\exists w \psi', \emptyset) \in S_{n+1}$. Finally, by the remark in the previous paragraph, this gives us $(\exists w \psi', \emptyset) \in S_{k+1}$, and consequently $(\psi', \gamma) \in S_{k+1}$ for some valuation $\gamma$ defined only on $w$. Then, again using the regularity axiom I, we conclude that $(\psi'[\gamma],\emptyset) \in S_{k+1}$ and $(\psi'[\gamma], \emptyset) \in S_{k+2}$. Since $\psi'[\gamma] = \psi[\gamma] = \psi(\num{x})$ for some $x$ from $M_{k}$ or $M_{k+1}$, we conclude that $\psi(\num{x}) \in T''$.

Conversely, suppose that $\phi(\num{x}) \in T''$ which means that $\phi(\num{x}) \in S_{n+1}$, where $n$ is the least such that $\phi(\num{x}) \in T''$. By regularity and compositional axioms this implies that we have $(\exists v \phi, \emptyset) \in S_{n+1}$. Then $(\exists v \phi, \emptyset) \in S_{k+1}$ where $k$ is the least such that $\exists v \phi \in M_k$ which again implies that $\exists v \phi \in T''$.

The regularity axiom of $\CT^-$ follows from the regularity axiom II in the above construction. This ends the proof modulo the consistency of the theory $\Theta_n$ which we prove in a separate lemma.
\end{proof}

\begin{lemma} \label{lemat_theta_n_sa_niesprzeczne}
	The theories $\Theta_n$ defined above are consistent. 
\end{lemma}
\begin{proof}[Sketch of the proof]
	We prove the claim by induction on $n$. Since the induction step and the initial step are essentially the same, we assume that $n>0$. There is only one additional thing which needs to be taken care of in the initial step and we will point it out in the construction. Suppose that $(M_n,T_n, J_n, S_n)$ satisfies $\Theta_{n-1}$. Notice that compositionality and preservation conditions are given by schemes:
	\begin{itemize}
		\item $\Comp_n(\phi)$, for $\phi \in \Form_{\PA}(M_n)$.
		\item $S_{n+1}(\phi,\alpha)$ for all $\phi \in M_{n-1}, \alpha  \in M_n$ such that $S_n(\phi,\alpha)$ holds.
	\end{itemize}
	  To prove the consistency of $\Theta_{n}$, take any finite $\Gamma \subset \Theta_{n}$. We want to interpret $S_{n+1}$ in the model $(M_n,T_n, J_n,A_n)$ so that it satisfies the finitely many compositional and preservation conditions from $\Gamma$. We will introduce an equivalence relation $\approx$ defined as follows for arithmetical formulae $\phi, \psi \in M_n$ and $\alpha \in \Val(\phi), \beta \in \Val(\psi)$:
	  \begin{displaymath}
	  (\phi,\alpha) \approx (\psi,\beta)
	  \end{displaymath}
	  if $\phi[\alpha]$ and $\psi[\beta]$ differ only by substituting a sequence of terms with equal values, i.e. there exists a formula $\xi \in M_n$ and sequences $\bar{t}, \bar{s} \in M_n$ of closed terms with $\bar{\val{s}}=\bar{\val{t}}$ such that $\phi[\alpha] = \xi(\bar{s})$ and $\psi[\beta] = \xi(\bar{t})$. For instance:
	  \begin{displaymath}
	  (\exists x \ x + (1\times1 +1)  = y, \alpha) \approx (\exists x \ x+ 2 \times z = u+1, \beta),
	  \end{displaymath}
	  where $\alpha(y) = 4, \beta(z) = 1, \beta(u )= 3$. 
	  
	  We also define a relation $\phi \approx \psi$ on formulae which holds if they are essentially the same up to substitution of terms. More precisely, for any formula $\phi$, define its \df{term trivialisation} $\widetilde{\phi}$ as the formula with smallest code such that
	  \begin{itemize}
	  	\item No constant symbol occurs in $\widetilde{\phi}$.
	  	\item No compound terms containing free variables occur in $\widetilde{\phi}$.
	  	\item No free variable occurs in $\widetilde{\phi}$ more than once.
	  	\item The formula $\phi$ can be obtained from $\widetilde{\phi}$ by substituting terms in such a way that variables in substituted terms will remain free after substitution. 
	  \end{itemize}
	   
	  For instance, if $\phi = \exists x (x+2 = 2 \times ((0 \times y) + S(0 +x)) + (y+z))$, then $\widetilde{\phi}$ is the following formula:
	  \begin{displaymath}
	  \exists x (x + v_0 = v_1 \times (v_2 + S(v_3 + x)) + v_4),
	  \end{displaymath}
	  where $v_i$'s as chosen so as to avoid clashes and assure minimality of $\widetilde{\phi}$. Observe that $\widetilde{\phi}$ is universal in the sense that if $\phi = \xi (\bar{t})$ for some $\bar{t} \in \TermSeq_{\PA}$, then $\xi = \widetilde{\phi}(\bar{s})$ for some $\bar{s} \in \TermSeq_{\PA}$. Notice that this property could be used to define term trivialisation.
	  
	   Finally, we say that $\phi \approx \psi$ iff $\phi, \psi$ have the same term trivialisation. The relation $\approx$ is clearly an equivalence relation. Notice that for any $\phi, \psi, \alpha, \beta$ if $(\phi, \alpha) \approx (\psi,\beta)$, then by definition $\phi[\alpha]$ and $\psi[\beta]$ can be obtained by substituting terms in the same formula $\xi$. But then $\xi = \widetilde{\phi} (\bar{s}) = \widetilde{\psi}(\bar{t})$ for some sequences of terms $\bar{s}, \bar{t} \in \TermSeq_{\PA}$ (not necessarily closed) and consequently $\widetilde{\phi} = \widetilde{\xi} = \widetilde{\psi}$.

	  Let $\Delta'$ be the finite set of all formulae which occur in $\Gamma$ under the predicate $S_{n+1}$ either in an instance of the compositionality scheme or the preservation condition. Let $\Delta$ be the set of equivalence classes of formulae from $\Delta'$ under the relation $\approx$:
	  \begin{displaymath}
	  \Delta = \set{[\phi]_{\approx} \in \Form_{\PA}(M_n)/ \approx}{\phi \in \Delta'}.
	  \end{displaymath} 
	  Notice that we can order $\Delta$ by the relation $\unlhd$ such that $[\phi] \unlhd' [ \psi]$ if there exist $\phi'  \in [\phi], \psi' \in [\psi]$ such that $\phi'$ is a direct subformula of $\psi$. Let $\unlhd$ be the transitive closure of $\unlhd'$. Now, we define the predicate $S_{n+1}$ in the following steps:
	  \begin{enumerate}
	  	\item In the first step, we include in $S_{n+1}$ all pairs $(\phi,\alpha)$ from $T_{n}$.
	  	\item For any $[\phi] \in \Delta$ which has nonempty intersection with $M_{n-1}$ and is minimal in the ordering $\unlhd$, we set $(\phi,\alpha) \in S_{n+1} $ iff $(\widetilde{\phi},\beta) \in S_n$ for some $\beta$ such that $(\phi,\alpha) \approx (\widetilde{\phi},\beta)$. Note that all the formulae in $[\phi]$ have the same trivialisation, so by elementarity $\widetilde{\phi} \in M_{n-1}$, since it is definable in a parameter from $M_{n-1}$. 
	  	\item For any $[\phi] \in \Delta$ which has no element in $M_{n-1}$ and is minimal in the ordering $\unlhd$, we do not add any $(\phi,\alpha)$ to $S_{n+1}$. Effectively, $\phi$ defines the empty set under the satisfaction predicate.
	  	\item If $n=0$, for all $\phi \in \Delta'$ which are subformulae of $\eta_b$ located on a (standard) finite depth in the syntactic tree of $\eta_b$ (including $\eta_b$ itself), we set $(\phi,\alpha) \in S_n$ if $A(\eta_b(\num{x}))$ holds where $x = \alpha(v)$. In effect, we decide that the valuations of all variables other than $v$ do not influence the truth value of $\eta_b$. If $n>0$, then $S_{n+1}$ is defined on $\eta_b$ and its direct subformulae by the preservation conditions.
	  	\item We extend the valuation to other classes in $\Delta$ by induction on the finite partial order $\unlhd$ using compositional conditions, e.g. if $S_{n+1}$ is already defined on $\phi$ such that $[\phi] \in \Delta$, then we extend it to $\neg \phi$ with $[\neg \phi] \in \Delta$ so that $(\neg \phi, \alpha) \in S_n$ iff $(\phi,\alpha) \notin S_n$.
	  \end{enumerate}
  	It is clear that the constructed model satisfies the elementary diagram of $(M_n,T_n,J_n,A_n)$.
	Since the predicate $S_{n+1}$ was defined by induction on complexity of formulae according to compositional condition and since every formula has an unambiguous tree of direct subformulae, the compositional conditions are satisfied. The preservation conditions are satisfied since if a formula $\phi$ is an element of $M_{n-1}$, then its direct subformula must be an element of $M_{n-1}$ as well. Since compositional conditions uniquely determine the behaviour of $S_{n+1}$ on a given formula given its behaviour on direct subformulae, $S_{n+1}$ agrees with $S_n$ on every formula in $\Gamma$ which belongs to $M_{n-1}$. It is clear that $T_{n+1} \subseteq S_{n+1}$ and that $\eta_b$ defines exactly the set $A$.
	
	Let us check that the regularity conditions are satisfied. They are clearly satisfied for formulae $\phi$ such that $[\phi] \notin \Delta$. We prove by induction on the height in the order $\unlhd$ in $\Delta$ that for all $[\phi]$ and all $\phi' \in [\phi]$, the regularity conditions are satisfied. The claim clearly holds for all formulae in $[\phi]$, where $[\phi]$ is minimal in the order $\unlhd$ in $\Delta$. Take any class $[\phi] \in \Delta$. We want to check that regularity conditions are satisfied for formulae in $[\phi]$, provided that they are satisfied for their direct subformulae. We prove this claim by cases, considering various possible syntactic shapes of $\phi$. Let us analyse one example. Suppose that $\phi = \exists v \psi$ such that regularity conditions are satisfied for formulae in $[\psi]$.

	We consider the first axiom of regularity. Take any $\alpha \in \Val(\phi)$ and without loss of generality assume that the variable $v$ is not in the domain of $\alpha$. By definition $(\phi, \alpha) \in S_{n+1}$ iff there exists $\alpha' \sim_v \alpha$ such that $(\psi, \alpha') \in S_{n+1}$. Notice that $(\psi, \alpha') \approx (\psi[\alpha],\beta)$, where $\beta$ is any valuation with $\beta(v) = \alpha'(v)$, as $\psi[\alpha]$ is a formula with at most the variable $v$ free and all other variables 'filled in' with $\alpha$.  By induction hypothesis,  $(\psi, \alpha') \in S_{n+1}$ if and only if $(\psi[\alpha],\beta)$ is in $S_{n+1}$. This in turn holds if and only if $(\phi[\alpha], \emptyset) \in S_{n+1}$, again by compositional conditions.
	 
	Now consider the second axiom of regularity. Let $\phi = \exists v \psi$, let $\bar{s}, \bar{t}$ be two coded sequences of closed terms with $\bar{\val{s}} = \bar{\val{t}}$  and suppose that $(\phi(\bar{s}), \emptyset) \in S_{n+1}$. Then there exists $\alpha \sim_{v} \emptyset$ such that $(\psi(\bar{s}), \alpha) \in S_{n+1}$. By assumption $(\psi(\bar{s}),\alpha) \approx (\psi(\bar{t}), \alpha)$, so by induction hypothesis $(\psi(\bar{t}), \alpha)  \in S_{n+1}$, and by compositional conditions $(\phi(\bar{t}), \emptyset) \in S_{n+1}$ as well.
	
	Similarly, the regularity conditions hold for all formulae from the classes in $\Delta$. This shows that the defined model satisfies the finite fragment $\Gamma$ of $\Theta_n$. The consistency of $\Theta_{n}$ follows.
\end{proof}

Let us comment on how to modify the proof of Lemma \ref{lem_rozszerzenia_ctminus_o_elastyczne_formuly} so that the constructed predicate satisfies $\SRP$. Rather than working with the equivalence classes of the $\approx$ relation considered in the proof, we work with a coarser \df{structural similarity} relation $\sim$. We define our satisfaction predicate simultaneously on all $\sim$-equivalent formulae and we require that the constructed satisfaction predicate is compatible with the relation of \df{structural equivalence} defined on pairs of formulae and valuations. Now we will define both relations, but we will first need some additional technical preliminaries.

\begin{definition} \label{def_structural_template}
	Let $\phi$ be an arithmetical formula. We say that $\widehat{\phi}$ is the \df{structural template} of $\phi$, if it is the smallest formula satisfying the following conditions:
	\begin{itemize}
		\item There exists a sequence $\bar{s}$ of terms such that $\phi$ and $\widehat{\phi}(\bar{s})$ differ by renaming bound variables in such a way that distinct variables remain distinct.
		\item Every free variable occurs in $\widehat{\phi}$ at most once. 
		\item No variable occurs in $\widehat{\phi}$ both free and bound.
		\item No closed terms occur in $\widehat{\phi}$. 
		\item No terms occur in $\widehat{\phi}$ whose all variables are free. 
	\end{itemize}	
	If $\phi$ and $\psi$ have the same structural template, we say that they are \df{structurally similar} and denote it with $\phi \sim \psi$. 
\end{definition}

\begin{example}
	For instance, if $\phi$ is 
	\begin{displaymath}
	x=y \wedge \exists x \exists y \bigl( x+ (x \times 0) = (z + S(z)) + y \times y) \bigr), 
	\end{displaymath}
	then its structural template $\widehat{\phi}$ is the following formula:
	\begin{displaymath}
	v_0=v_1 \wedge \exists w_1 \exists w_2 \bigl( w_1+ (w_1 \times v_2) = v_3 + w_2 \times w_2) \bigr), 
	\end{displaymath}
	where $v_i, w_i$ are chosen so as to guarantee minimality. 
\end{example}

\begin{example}
	The following formulae $\phi_1,\phi_2$ are structurally similar:
	\begin{eqnarray*}
		\phi_1 & = & \forall x \exists y  \big( x+y = S(0) \times z )\big) \\
		\phi_2 & = & \forall w \exists u \big( w+u = S(x + 0) )\big).
	\end{eqnarray*}
\end{example}

Finally, we can define the structural equivalence relation. We say that for $\phi,\psi \in \Form_{\PA}$, $\alpha \in \Val(\phi), \beta \in \Val(\psi)$, the pairs $(\phi,\alpha)$ and $(\psi, \beta)$ are \df{structurally equivalent} if  
\begin{displaymath}
\phi[\alpha] \simeq \psi[\beta]
\end{displaymath}
in the sense of Definition \ref{def_structural_equivalence}. In the construction of a satisfaction predicate satisfying $\SRP$, we require that $S(\phi,\alpha) \equiv S(\psi,\beta)$ holds whenever $(\phi,\alpha)$ and $(\psi,\beta)$ are structurally equivalent.

\bibliographystyle{abbrvnat}

\bibliography{Disjunctions} 

\end{document}